\documentclass[11pt]{article}

\usepackage{fixed_locus}

\title{Structure of the Components of the Fixed Locus \\ of a Self-Map of the
  Berkovich Line}

\author{Xander Faber \\
  IDA / Center for Computing Sciences \\
Bowie, MD \\
xander@super.org
\and
Niladri Patra \\
Indian Statistical Institute, Delhi Centre, \\
S. J. S. Sansanwal Marg, New Delhi, India. \\
niladri@math.tifr.res.in}

\begin{document}
\maketitle

\begin{abstract}
  We describe the local and global structure of the fixed locus for the action
  of a rational function on the Berkovich projective line over a complete
  nontrivially-valued algebraically closed nonarchimedean field. This includes
  a bound for the number of connected components that is sharp when the residue
  characteristic is large or zero. The case of small nonzero residue
  characteristic will be treated in a subsequent article.
\end{abstract}

\tableofcontents


\section{Introduction}

In the study of the iteration of a holomorphic function $\varphi$ on the
Riemann sphere, the local behavior near fixed points, and by extension periodic
points, dictates a great deal about the global dynamics; see
\cite[\S8-14]{Milnor_Dynamics_Book_2006} and \cite{McMullen_Isospectral}. The
same is true when we consider the dynamics of a rational function over a
nontrivially-valued nonarchimedean field $K$, but there are some new
wrinkles. One that has been addressed at length over the past two decades is
the fact that the natural topology of nonarchimedean fields is sufficiently
unpleasant that one should enrich the standard projective line $\PP^1(K) = K
\cup \{\infty\}$ to obtain the Berkovich projective line $\BerkK$, which is
compact, Hausdorff, and uniquely arcwise connected
\cite{Baker-Rumely_BerkBook_2010,
  Benedetto_dynamics_book,Rivera-Letelier_Asterisque_2003}.

Shifting to this improved topological setting introduced its own new wrinkle,
though. If the fixed points of a (complex) holomorphic function $\varphi(z)$
have an accumulation point in the domain of definition, then $\varphi(z) =
z$. This is the entire basis for analytic continuation. But a rational function
$f \in K(z)$, viewed as a self-map of $\BerkK$, is not so constrained: the set
of fixed points of $f$ can contain closed arcs or open neighborhoods for the
strong topology. Our goal in the present note is to bring some order to this
new chaos. We characterize and bound the components of the fixed locus, and we
prove rigidity results that describe the internal structure and the
$K$-rational fixed points of a given component. This work and its sequel
\cite{Faber_Patra_fixed_bound} constitute a massive extension of the second
author's paper on the fixed locus for a function with potential good reduction
\cite{Patra_fixed}.


\subsection{Results}

Let $K$ be an algebraically closed field that is complete with respect to a
nontrivial nonarchimedean absolute value. We write $\BerkK$ for the analytic
projective line over $K$ in the sense of Berkovich, and we write $f \in K(z)$
for a nonconstant rational function, which we view as a self-map
${f\colon\BerkK\to\BerkK}$. All topological statements about $\BerkK$ refer to
the weak topology, unless otherwise specified.

Given a nonconstant rational function $f \in K(z)$, the \textbf{fixed locus} of
$f$ is defined to be the set
\[
  \Fix(f) = \{x \in \BerkK : f(x) = x\}.
  \]
Evidently, $\Fix(f)$ is a closed subset of $\BerkK$. The primary goal of this
article is to describe the local and global structure of the connected
components of $\Fix(f)$.

\begin{definition}
  \label{def:component_types}
  Let $X$ be a component of the fixed locus of $f$.
  \begin{enumerate}
    \item $X$ is \textbf{classical} if it consists of a single classical
      (a.k.a. type~I) fixed point.
    \item $X$ is \textbf{indifferent} if every point of $X$ is indifferent; and
    \item $X$ is \textbf{peaked} if $X$ contains a repelling fixed point of
      type~II. 
  \end{enumerate}
\end{definition}

We argue in Proposition~\ref{prop:components} that each component of $\Fix(f)$
is closed and can be described by exactly one of the above types. The idea
behind the ``peaked'' terminology is that the points with local degree larger
than~1 must be isolated in the component, so they look like ``peaks'' when you
graph the local degree function.

The first main result of this paper shows that the number of classical fixed
points in a non-classical component depends only on the local behavior of $f$
at its type~II repelling fixed points.

\begin{theoremA}
  Let $f \in K(z)$ be a nonconstant rational function, and let $X$ be a
  non-classical component of $\Fix(f)$. Then $X$ contains finitely many type II
  repelling fixed points, and the number of classical fixed points inside $X$,
  counted according to fixed-point multiplicity, is
  \[
     2 + \sum_{\substack{x \in X \\ \text{type~II repelling}}} \left(\deg_f(x) - 1 - \Ncf(f,x)\right),
  \]
  where we write $\Ncf(f,x)$ for the number of critically fixed directions of
  the tangent map $T_x f$.
\end{theoremA}

\begin{remark}
  \label{rem:sum_over_all_pts}
  If $x$ is either a classical fixed point or a non-classical indifferent fixed
  point, then $\deg_f(x) - 1 = \Ncf(f,x) = 0$. So we may as well sum over all
  $x \in X$ in Theorem~A.
\end{remark}

Theorem~A implies that an indifferent component contains exactly two classical
fixed points, counted with multiplicity. As $f$ has precisely $d+1$ classical
fixed points, we see that $\Fix(f)$ admits at most $\frac{d+1}{2}$ indifferent
components. We are unsure if this bound is sharp. 

In \cite{Rumely_new_equivariant}, Rumely defined the ``crucial weight''
$\wR(x)$ of a point $x \in \BerkK$ relative to a rational function $f$; we
recall its definition in \S\ref{sec:crucial_weight}. Rumely proved a striking
formula for the total crucial weight of all points:

\begin{theorem}[Rumely's Weight Formula {\cite[Thm.~6.1]{Rumely_new_equivariant}}]
  If $f \in K(z)$ is a rational function of degree $d \ge 1$, then
  \[
    \sum_{x \in \BerkK} \wR(x) = d - 1.
  \]
\end{theorem}

Rumely's original proof proceeds through a careful study of the Laplacian of
the function $\mathrm{ordRes}_f : \BerkK \to [0,\infty]$. Theorem~A enables us
to give a new proof of Rumely's Weight Formula that avoids potential theory
entirely.

It is well known that a rational function $f$ of degree $d \ge 1$ has exactly
$d+1$ $K$-rational fixed points when counted appropriately, but it is not at
all obvious whether the more general object $\Fix(f)$ even has finitely many
components. Rumely's Weight Formula and Theorem~A give an easy argument for a
weak bound:

\begin{corollary}
  \label{cor:weak_bound}
  The fixed locus of a rational function of degree $d \ge 1$ has at most $2d$
  connected components.
\end{corollary}

\begin{proof}
  There are at most $d+1$ components that contain a classical fixed
  point. Every indifferent component contains a classical fixed point by
  Theorem~A, and every peaked component contains a type~II repelling fixed
  point by definition. Rumely's Weight Formula implies that there are at most
  $d-1$ peaked components since a repelling type~II fixed point has positive
  crucial weight. Thus, there are at most $(d+1) + (d-1) = 2d$ components of
  $\Fix(f)$.
\end{proof}

With an extra assumption on the residue characteristic of $K$, we can prove
something much stronger.

\begin{theoremB}
  Write $p \ge 0$ for the residue characteristic of $K$ and let $f \in K(z)$ be
  a nonconsant rational function of degree $d \ge 1$. Suppose that either $p =
  0$ or $p > d$. Then every connected component of $\Fix(f)$ contains a
  classical fixed point. In particular, the fixed locus of $f$ has at most
  $d+1$ connected components.
\end{theoremB}


In the sequel to this paper \cite{Faber_Patra_fixed_bound}, we improve the
bound in Corollary~\ref{cor:weak_bound} when the residue characteristic $p$ is
at most the degree~$d$: the number of components of $\Fix(f)$ is at most $d +
\frac{d}{p} + 1$. 

The structure of this article is as follows.  In
Section~\ref{sec:preliminaries}, we will prove some elementary topological
results on connected subsets of $\BerkK$ and $\Fix(f)$, and then recall the
necessary background on surplus multiplicities and classical fixed points. In
Section~\ref{sec:indifferent}, we discuss the local structure of a component of
$\Fix(f)$ near an indifferent fixed point. We then explicitly describe the
locus of fixed points for an invertible map in Section~\ref{sec:invertible}.
In Section~\ref{sec:indifferent_components}, we look at the global structure of
indifferent components. We prove Theorem~A in
Section~\ref{sec:counting_classical}. Section~\ref{sec:hyperbolic} is devoted
to studying hyperbolic components of $\Fix(f)$ --- i.e., those that contain no
classical fixed point. Finally, we derive Rumely's Weight Formula and some
consequences in Section~\ref{sec:crucial_weight}, including Theorem~B.


\subsection{Notation}

If $K$ is an algebraically closed field that is complete with respect to a
nontrivial nonarchimedean absolute value $|\cdot|$, we write $\cO$ for its
valuation ring, $\mm$ for the maximal ideal of $\cO$, and $\tilde K = \cO /
\mm$ for the residue field. For an element $a \in \cO$, we write $\tilde a$ for
its image in $\tilde K$. For a polynomial $g \in \cO[z]$, we write $\tilde g$
for its image in $\tilde K[z]$. Similarly, if $f = g / h \in K(z)$ is a
rational function with $g, h \in \cO[z]$, we write $\tilde f = \tilde g /
\tilde h$, which is well defined up to multiplication by a unit in $\cO$.

For the Berkovich line $\BerkK$, we borrow most of our notation from
\cite{Baker-Rumely_BerkBook_2010}; see \cite[\S2]{Faber_Berk_RamI_2013} for a
brief synopsis. We draw attention to a few small deviations that have either
become more common in other writings, or which we have found useful in the
present text. Let $f \in K(z)$ be a nonconstant rational function. We write
$\deg_f(x)$ for the local degree of $f$ at $x$; this corresponds to the
``multiplicity'' $m_f(x)$ in \cite{Baker-Rumely_BerkBook_2010}. If $T_x$ is the
space of directions at $x \in \BerkK$, we write $T_x f : T_x \to T_{f(x)}$ for
the induced map on directions. By definition, a direction $\vv \in T_x$ is a
connected component of $\BerkK \smallsetminus \{x\}$; we write $B_x(\vv)^-$
instead of $\vv$ when we wish to discuss membership in this open ball.


\section{Preliminaries}
\label{sec:preliminaries}

We begin this section with some generalities on connected subsets of
$\BerkK$. Then we justify the classification of the connected components of
$\Fix(f)$ from Definition~\ref{def:component_types}. We close by collecting
several useful results regarding the surplus multiplicity and fixed points.


\subsection{Connected subsets of \texorpdfstring{$\BerkK$}{the analytic line}}

Let $T$ be a Hausdorff topological space.  An \textbf{arc} in $T$ is the image
of a continuous injective map $s : [0,1] \to T$. Here $s(0)$ and $s(1)$ are the
\textbf{endpoints} of the arc.  We say that $T$ is \textbf{arcwise connected}
if for any two distinct points $x,y \in T$, there is an arc with endpoints $x$
and $y$; such an arc is said to \textbf{join $x$ and $y$}. The space $T$ is
\textbf{uniquely arcwise connected} if there is a unique arc joining any pair
of distinct points. If $T$ is uniquely arcwise connected, we write $[x,y]$ for
the arc joining a pair of distinct points $x,y \in T$. We will have occasion to
abuse this terminology by allowing $x = y$, and saying that $[x,x] = \{x\}$ is
an arc joining $x$ to itself.

\begin{remark}
  A subset $X$ of a Hausdorff space $T$ is path-connected if and only if it is
  arcwise connected \cite[Cor.~31.6]{willard_topology}.
\end{remark}

\begin{proposition}
  \label{prop:closest_point}
  Let $T$ be a Hausdorff and uniquely arcwise connected space.  Let $y \in T$,
  and let $X \subset T$ be a nonempty path-connected closed subset. There is a
  unique point $c \in X$ with the following property: for any $x \in X$, the
  arc $[y,x]$ contains $c$. Moreover, if $y \not\in X$, then $c$ lies in the
  boundary of $X$.
\end{proposition}

In the setting of the proposition, we will say $c$ is \textbf{the point of $X$
  that is closest to $y$}.


\begin{proof}[Proof of Proposition~\ref{prop:closest_point}]
  Suppose first that $y \in X$. Set $c = y$. Evidently, every arc $[y,x]$
  contains $c$, and taking $x = y$ shows uniqueness.

  Now suppose that $y \not\in X$. Let $x \in X$, and let $s : [0,1] \to T$ be a
  continuous injection with $s(0) = y$ and $s(1) = x$. As $s^{-1}(X)$ is
  closed, and hence compact, there exists an element
  \[
  m_s := \min\left(s^{-1}(X)\right) \in [0,1].
  \]
  Note $m_s \ne 0$ since $y \not\in X$. Define $c = s(m_s)$. By construction,
  $c \in [y,x] = \im(s)$. 

  We now argue that $c$ lies on every arc from $y$ to a point of $X$. Let $x'
  \in X$, and let $t : [0,1] \to T$ be a continuous injection with $t(0) = y$
  and $t(1) = x'$. Define
  \[
     m_t := \min\left(t^{-1}(X)\right),
  \]
  and set $c' = t(m_t)$. For the sake of a contradiction, assume that $c \ne
  c'$. Since $X$ is path-connected, there is a continuous injection $u : [0,1]
  \to X$ with $u(0) = c'$ and $u(1) = c$. Define $v : [0,1] \to T$ by
  \[
  v(a) = \begin{cases}
    t(2 m_t \, a) & \text{ if } 0 \le a < 1/2 \\
    u(1 + 2(a-1)) & \text{ if } 1/2 \le a \le 1 \\
  \end{cases}
  \]
  By construction, $v$ is continuous. It is injective since each piece of it is
  injective, and since the image of the first part lies outside $X$ while the
  image of the second half lies entirely inside $X$. The image of $v$ is the
  arc $[y,c]$, and it contains the arc $[c,c']$. In particular, $[y,c]$
  contains at least two points of $X$: $c$ and $c'$. The image of
  $s|_{[0,m_s]}$ is also the arc $[y,c]$, but the only point of $X$ in this
  image is $c$. As $T$ is uniquely arcwise connected, we have arrived at a
  contradiction. Thus, we are forced to conclude that $c = c'$, and $c$ lies on
  every arc from $y$ to a point of $X$.

  Note that the image of the map $s|_{[0,m_s]}$ is the arc $[y,c]$, and $X \cap
  [y,c] = \{c\}$. This means $c$ is the unique point that lies on every arc
  from $y$ to a point of $X$.
  
  Let $U$ be a neighborhood of $c$. The inverse image of $U$ under the map
  $s|_{[0,m_s]}$ is an open subset of $[0, m_s]$ containing $m_s$. As $U$ was
  chosen arbitrarily, every neighborhood of $c$ contains points of the
  half-open arc $[y,c)$, which lies outside $X$. It follows that $c$ lies in
    the boundary of $X$.
\end{proof}

The projective line $\BerkK$ is uniquely arcwise connected
\cite[Lem.~2.10]{Baker-Rumely_BerkBook_2010} and Hausdorff
\cite[Prop.~2.6]{Baker-Rumely_BerkBook_2010}. It follows that
Proposition~\ref{prop:closest_point} applies to path-connected closed subsets
of $\BerkK$; we will use this observation without comment in the remaining
sections.

We will also conflate the notion of ``component'' with ``path component'' for
subsets of $\BerkK$. This is justified by
\cite[Prop.~3.9]{Rivera-Letelier_Espace_Hyperbolique_2003} and the fact that a
subset of $\BerkK$ is (arcwise) connected for the weak topology if and only if
it is (arcwise) connected for the strong topology \cite[Lem.~B.18,
  Cor.~B.21]{Baker-Rumely_BerkBook_2010}.


\subsection{Connected components of \texorpdfstring{$\Fix(f)$}{the fixed locus}}
\label{sec:components}

Recall  that the fixed locus of a rational function $f \in
K(z)$ is defined to be the set
\[
  \Fix(f) = \{x \in \BerkK : f(x) = x\}.
  \]

\begin{proposition}
  Let $f \in K(z)$ be a nonconstant rational function. The fixed locus is a
  closed set, and each connected component of it is closed. 
\end{proposition}

\begin{proof}
  Observe that the diagonal $\Delta \subset \BerkK \times \BerkK$ is closed
  since $\BerkK$ is Hausdorff. The pre-image of the diagonal under the graph
  $\Gamma_f : \BerkK \to \BerkK \times \BerkK$ is the fixed locus. This proves
  the first statement. The second statement follows from the first and the fact
  that connected components are closed \cite[Thm.~26.12]{willard_topology}.
\end{proof}

Let $x$ be a classical (type I) fixed point.  The \textbf{fixed-point
  multiplier} (or multiplier) of $x$ is $\lambda_x := Df(x) \in K^\times$. We
say that $x$ is attracting, repelling, or indifferent depending on whether
$|\lambda_x|$ is less than, greater than, or equal to~1, respectively. A
non-classical fixed point $x$ is repelling or indifferent depending on whether
the local degree $\deg_f(x)$ is greater than or equal to 1, respectively. Every
type~III or type~IV fixed point of $f$ is indifferent
\cite[Lem.~10.80]{Baker-Rumely_BerkBook_2010}.

\begin{lemma}
  \label{lem:type_I}
  Let $f \in K(z)$ be a nonconstant rational function. If $x$ is a classical
  fixed point of $f$, then $x$ is isolated in $\Fix(f)$ if and only if $x$ is
  attracting or repelling.
\end{lemma}

\begin{proof}
  Conjugate $f$ so that $x = 0$. Near $0$, we have
  \begin{equation}
    \label{eq:phi_series}
    f(z) = \lambda z + a_2 z^2 + a_3 z^3 + \cdots. 
  \end{equation}

  Suppose that $\lambda = 0$. Let $\ell \ge 2$ be the minimum index such that
  $a_\ell \ne 0$. Let $y \in K$, and write $\varepsilon = |y|$. For
  $\varepsilon$ sufficiently small, we have
  \[
    |f(y)| = |a_\ell| \varepsilon^\ell < \frac{1}{2} |y|^{\ell-1} < |y|.
  \]
  It follows that $f^n(\zeta) \to 0$ for every $\zeta$ in a neighborhood of
  $0$, which means $0$ is an isolated fixed point.

  Now suppose that $|\lambda| > 0$ and let $|y| = \varepsilon > 0$. For
  $\varepsilon$ sufficiently small, we have
  \begin{equation}
    \label{eq:mult_scaling}
      |f(y)| = |\lambda| \varepsilon. 
  \end{equation}
  If $|\lambda| \ne 1$, we see that $|f(y)| \ne |y|$. In particular, this means
  there is a neighborhood of $0$ in $\BerkK$ in which $0$ is the only fixed point. This
  concludes the proof that $x$ is isolated if it is attracting or repelling. If
  instead $x$ is an indifferent fixed point --- i.e., $|\lambda| = 1$ --- then
  \eqref{eq:mult_scaling} shows that the type~II point $\zeta_{0,\varepsilon}$
  is fixed. Letting $\varepsilon \to 0$, we see that $0$ is not an isolated
  fixed point when $x$ is indifferent.
\end{proof}

We now justify the terminology introduced in Definition~\ref{def:component_types}. 

\begin{proposition}
  \label{prop:components}
  Let $f \in K(z)$ be a nonconstant rational function. Each component of
  $\Fix(f)$ is either classical, indifferent, or peaked. A classical component
  consists of an attracting or repelling fixed point.
\end{proposition}

\begin{proof}
  Let $X$ be a component of $\Fix(f)$. If $X$ contains an attracting or
  repelling classical fixed point, then $X = \{x\}$ is a classical component
  (Lemma~\ref{lem:type_I}). If $X$ contains a repelling type~II point, then it
  is a peaked component. We are left with the case where every type~I or~II
  point of $X$ is indifferent. Since every type~III or type~IV fixed point of
  $f$ is indifferent, we conclude that $X$ is an indifferent component.
\end{proof}

\begin{remark}
  \label{rem:higher_type_fixed}
  Every indifferent component contains a type~II fixed point by the proof of
  Lemma~\ref{lem:type_I}.
\end{remark}


\subsection{Surplus multiplicities and fixed points}
\label{sec:rumely}

The next result gives a local criterion for when a classical fixed point lies
in a particular direction from a type II fixed point.  To set up the result, we
need additional notation. Let $f \in K(z)$ be a nonconstant rational function,
let $x \in \BerkK$ be fixed by $f$, and let $\vv \in T_x$ be a direction.
\begin{itemize}
  \item Write $s_f(\vv)$ for the surplus multiplicity of $f$ in the direction
    $\vv$. (See \cite[\S3.3]{Faber_Berk_RamI_2013}.)
  \item Write $F_f(\vv)$ for the number of classical fixed points in the
    direction $\vv$.
  \item Write $\tilde F_f(\vv)$ for the fixed-point multiplicity of $\vv$ under
    the tangent map $T_x f$. (We set $\tilde F_f(\vv) = 0$ if $\vv$ is not a
    fixed direction.)
\end{itemize}

\begin{lemma}[First Identification Lemma {\cite[Lem.~2.1]{Rumely_new_equivariant}}]
  Let $f \in K(z)$ be nonconstant. Suppose $x \in \BerkK$ is a fixed point of
  type~II, and that $T_x f$ is not the identity. For $\vv \in T_x$, we
  have
  \[
     F_f(\vv) = s_f(\vv) + \tilde F_f(\vv).
     \]
  In particular, the open ball $B_x(\vv)^-$ contains a classical fixed point of
  $f$ if and only if $s_f(\vv) > 0$ or $\tilde F_f(\vv) > 0$.
\end{lemma}


The simplest method for computing the surplus multiplicities at a type~II point
$\zeta$ is to change coordinate so that $\zeta = \zeta_{0,1}$, and then to look
at the factors that cancel from the numerator and denominator of $f$ upon
reducing the coefficients \cite[Lem.~3.17]{Faber_Berk_RamI_2013}. A second
approach is via the ``jump-number formula'', which we recall now.

\begin{jumpnumber}[{\cite[Prop.~3.18]{Faber_Berk_RamI_2013}}]
  Let $f \in K(z)$ be a nonconstant rational function. For a point $y \in
  \BerkK$ and a tangent vector $\ww \in T_y$, write $\deg_f(y,\ww)$ for the
  local degree of $T_y f$ in the direction $\ww$. If $y \in B_x(\vv)^-$, write
  $\ww_x$ for the unique tangent vector in $T_y$ that points back toward
  $x$. Then 
  \[
     s_f(\vv) = \sum_{y \in B_x(\vv)^-} \Big(\deg_f(y) -
     \deg_f(y,\ww_x)\Big).
     \]
\end{jumpnumber}

The nonzero contributions in the Jump Number Formula occur precisely at points
where the local degree increases, as viewed from the point $x$. 

\begin{remark}
The formula in \cite{Faber_Berk_RamI_2013} contains the summand
$\max\{\deg_f(y) - \deg_f(y,\ww_x), 0\}$, but the first author failed to
observe that the difference in this maximum is always nonnegative.
\end{remark}


\section{Indifference lemmas}
\label{sec:indifferent}

We now provide a finer understanding of the local structure of non-classical
components through a sequence of four lemmas. The first addresses the
structure near a classical indifferent fixed point.

\begin{lemma}[First Indifference Lemma]
    \label{lem:first_indifference_lemma}
  Let $f \in K(z)$ be a rational function. Suppose that $x$ is an indifferent
  classical fixed point of $f$ with multiplier $\lambda$. There is a point $x'$
  such that every $y \in (x,x')$ is fixed and indifferent, and if $y$ is of
  type~II, then $T_x f(z) = \tilde \lambda z$ in an appropriate coordinate.
\end{lemma}

\begin{proof}
  If necessary, we make a change of coordinate in order to assume that $x =
  0$. The function $f(z) - \lambda z$ has a zero of order $m \ge 2$ at the
  origin, so we have
  \[
     f(z) = \lambda z + z^m \frac{g(z)}{h(z)}
  \]
  for some polynomials $g, h \in K[z]$ with $g(0)h(0) \ne 0$. For $u \in
  K^\times$, define
  \[
     f_u(z) = u^{-1}f(uz) = \lambda z + u^{m-1}z^{m} \frac{g(uz)}{h(uz)}.
  \]
  If $|u|$ is sufficiently small, then all of the nonconstant coefficients of
  $g(uz)$ and $h(uz)$ are integral and $|u|^{m-1} g(0)/h(0) < 1$. Moreover, the
  coefficient of $z^{m}$, which is $u^{m-1} g(0)/h(0)$, dominates all of the
  coefficients of higher degree terms in the series expansion of $f_{u}$. For
  such $u$, it follows that $\tilde f_u(z) = \tilde \lambda z$. This means $f$
  fixes the point $x = \zeta_{0,\varepsilon}$ for all sufficiently small positive
  numbers $\varepsilon$, and the tangent map at a type II point of this form is
  $T_x f(z) = \tilde \lambda z$ in the coordinate specified by $f_u$.
\end{proof}

The Second Indifference Lemma addresses the local structure in a non-critical
fixed direction at a type II fixed point.

\begin{lemma}[Second Indifference Lemma]
    \label{lem:second_indifference_lemma}
  Let $f \in K(z)$ be a rational function. Let $x$ be a type~II fixed
  point for $f$, and let $\vv$ be a fixed direction for $T_x f$ with multiplier
  $\tilde \lambda \ne 0$. There is a fixed point
  $x' \in B_x(\vv)^-$ such that every $y \in (x,x')$ is fixed and indifferent,
  and if $y$ is of type~II, then in an appropriate coordinate $T_y f(z) =
  \tilde \lambda z$.
\end{lemma}

\begin{proof}
  If necessary, we make a change of coordinate so that $x = \zeta_{0,1}$, $\vv
  = \vec{0}$, and $\vec{\infty}$ is not a bad direction (i.e., a direction
  $\vv$ with $s_f(\vv) > 0$). Write
  \[
    f = \frac{a_d z^d + \cdots + a_0}{b_d z^d + \cdots + b_0},
  \]
  where $a_i, b_i \in \cO_K$. After a suitable rescaling of the coefficients,
  we may assume there exists an index $i$ such that $a_i \in \cO_K^\times$ or
  $b_i \in \cO_K^\times$. Writing $s = s_f(\vec{0})$ for the surplus
  multiplicity for the direction $\vec{0}$, we find that $|a_i| < 1$ and $|b_i|
  < 1$ for $i = 0,\ldots, s-1$, and $a_s \in \cO_K^\times$ or $b_s \in
  \cO_K^\times$. In fact, since $\vec{0}$ is fixed, we see that $b_s \in
  \cO_K^\times$ and $|a_s| < 1$. The multiplier at $\vec{0}$ is $\tilde \lambda
  = \tilde a_{s+1} / \tilde b_s$, so our hypothesis implies $a_{s+1} \in
  \cO_K^\times$.  Without loss of generality, we may rescale the coefficients
  of $f$ by $b_s^{-1}$ in order to assume that $b_s = 1$ and $\tilde a_{s+1} =
  \tilde \lambda$.

  Let $u \in \cO_K$ be nonzero, and define
  \[
  g(z) = u^{-1}f(uz) = \frac{a_d u^{d-s-1} z^d + \cdots + a_{s+1}z^{s+1}
    + a_s u^{-1}z^s + \cdots + a_0 u^{-s-1}}
  {b_d u^{d-s} z^d + \cdots + z^s + b_{s-1} u^{-1} z^{s-1} + \cdots + b_0 u^{-s}}.
  \]
  Set
  \[
  r_0 = \max\left(
  \left\{ |a_i|^{\frac{1}{s+1-i}} : i = 0, \ldots, s \right\} \cup
  \left\{ |b_j|^{\frac{1}{s-j}} : j = 0, \ldots, s-1 \right\}
  \right) < 1.
  \]
  If we assume that $r_0 < |u| < 1$, then all coefficients of $g$ are integral,
  and
  \[
     \tilde g(z) = \tilde a_{s+1}z = \tilde \lambda z.
  \]
  Thus, for all $r_0 < r < 1$ we have $f(\zeta_{0,r}) = \zeta_{0,r}$, and
  $T_{\zeta_{0,r}}f(z) = \tilde \lambda z$ when $\zeta_{0,r}$ is of type II.
\end{proof}

Before proceeding further, we need Rumely's classification of indifferent fixed
points \cite{Rumely_new_equivariant}. Suppose that $x$ is a type~II indifferent
fixed point of a rational function $f \in K(z)$. Then
\begin{samepage}
\begin{itemize}
  \item $x$ is \textbf{id-indifferent} if $T_x f$ is the identity map;
  \item $x$ is \textbf{multiplicatively indifferent} if $T_x f(z) = c z$ in some
    coordinate, with $c \ne 1$; and
  \item $x$ is \textbf{additively indifferent} if $T_x f(z) = z + 1$ in some
    coordinate.
\end{itemize}
\end{samepage}
We can extend this classification to type~III and type~IV fixed points by
embedding $\BerkK \hookrightarrow \BerkL$, where $L$ is a spherically complete
extension of $K$ such that $|L^\times| = \RR$. Then $\BerkL$ has only type~I
and type~II points. The map $f$ extends uniquely to $\BerkL$, and the properties of
$T_x f$ at a type~II point do not change under extension to $L$.

The Third Indifference Lemma shows that the indifference type is locally
constant for the strong topology, as long as we stay away from the leaves
(endpoints) of a component.

\begin{lemma}[Third Indifference Lemma]
    \label{lem:third_indifference_lemma}
  \label{prop:locally_constant}
  Let $f \in K(z)$ be a nonconstant rational function, let $X$ be a
  component of the fixed locus of $f$, and let $x \in X$ be of type~II.
  \begin{itemize}
    \item If $x$ is multiplicatively indifferent, there is a strong
      neighborhood $V$ of $x$ such that every point of $V \cap X$ is
      multiplicatively indifferent.
    \item If $x$ is id-indifferent, there is a strong neighborhood $V$ of $x$
      such that every point of $V$ is id-indifferent.
  \end{itemize}
\end{lemma}

\begin{proof}
  One can deduce this result directly as in the proof of the Second
  Indifference Lemma. However, we will take a more conceptual tack by using
  Rivera-Letelier's Lemme d'Approximation
  \cite[\S10]{Rivera-Letelier_Periodic_Points_2005}: \textit{If $f, g \in K(z)$
    are two rational functions that fix the Gauss point and satisfy $\tilde f =
    \tilde g$, then there exists a strong neighborhood $V$ of $\zeta_{0,1}$
    such that $f|_V = g|_V$.}

  Returning to our setting, without loss of generality, we make a change of
  coordinate so that $x = \zeta_{0,1}$ is the Gauss point, and $T_x f(z) =
  \tilde \lambda z$ for some $\lambda \in \cO_K^\times$. Set $g(z) = \lambda
  z$. Then $f$ and $g$ both fix the Gauss point and satisfy $\tilde f = \tilde
  g$. Thus, there exists a strong neighborhood $V$ of $x$ such that $g$ fixes
  every point $y \in V \cap X$ and $T_y f = T_y g$, which completes the proof.
\end{proof}

Our fourth and final indifference lemma describes a reciprocity between the
fixed-point multipliers of opposing directions along an arc in a fixed
component.

\begin{lemma}[Fourth Indifference Lemma]
    \label{lem:fourth_indifference_lemma}
   Let $f \in K(z)$ be a nonconstant rational function. Suppose that $x_1, x_2$
   are type~II fixed points such every point of the arc $(x_1,x_2)$ is fixed
   and indifferent. For $i=1,2$, let $\vv_i \in T_{x_i}$ be the direction
   containing the arc $(x_1,x_2)$, and let $\lambda_i$ be the fixed-point
   multiplier for $T_{x_i} f$ at $\vv_i$. Then
   \[
     \lambda_1 \lambda_2 = 1 \text{ in $\tilde K$}.
   \]
\end{lemma}

\begin{proof}
  Without loss of generality, we may change coordinate so that $x_1 =
  \zeta_{0,1}$ and $x_2 = \zeta_{0,R}$ for some $R > 1$. Every point
  $\zeta_{0,r}$ with $1 \le r \le R$ is fixed. It follows that if $u \in K$
  satisfies $1 \le |u| \le R$, then the function
  \[
     f_u(z) = u^{-1}f(uz)
  \]
  has nonconstant reduction. We will now compute this reduction.

  An argument just like the one in the proof of the Second Indifference Lemma
  shows that there is a positive number $\varepsilon_1 > 0$ such that for any
  $u \in K$ with $1 < |u| < \varepsilon_1$, we have $\tilde f_u(z) =
  \lambda_1^{-1} z$. (Note that the fixed-point multiplier at $\infty$ for the
  map $z \mapsto \lambda_1^{-1} z$ is $\lambda_1$.) Similarly, there is a
  positive number $\varepsilon_2 > 0$ such that for any $u \in K$ with $R -
  \varepsilon_2 < |u| < R$, we have $\tilde f_u(z) = \lambda_2 z$. Finally, the
  argument in the third indifference lemma shows that for each $r$ with $1 < r
  < R$, there is $\lambda(r) \in \tilde K^\times$ and an open interval $I_r
  \subset (1,R)$ such that
  \[
    |u| \in I_r \quad \Longrightarrow \quad \tilde f_u(z) = \lambda(r) z.
  \]
  Since the intervals $\{I_r : 1 < r < R\}$ cover $(1,R)$, we see that the
  function $r \mapsto \lambda(r)$ is constant. We simply write $\lambda \in
  \tilde K^\times$ for this constant value. Taking $r \in (1,1+\varepsilon)$
  shows that $\lambda = \lambda_1^{-1}$, while taking $r \in (R-\varepsilon_2,
  R)$ shows that $\lambda = \lambda_2$. It follows that $\lambda_1 \lambda_2 =
  1$ in $\tilde K$, as desired.
\end{proof}


\section{Invertible maps}
\label{sec:invertible}

Our goal for this section is to give a complete description of the fixed locus
for an invertible rational function. It is useful to note that the geometry and
topology of $\Fix(f)$ are unaffected by conjugation. Observe that the map $g(z)
= bz + a$ sends the Gauss point to the point $\zeta_{a, |b|}$. Thus a rational
function $f$ fixes $\zeta_{a, |b|}$ if and only if $g^{-1} \circ f \circ g$ fixes
the Gauss point, which is also equivalent to the map $g^{-1} \circ f \circ g$
having nonconstant reduction. Another helpful observation is that any type~II
point lying outside the arc $[0, \infty]$ is of the form $\zeta_{a, |b|}$, for
some $a,b \in K$ satisfying $|a| > |b| > 0$. We use these observations heavily
in the proof of the following theorem.

Define the \textbf{id-indifference locus} of $f$ to be
\[
  I_f = \{x \in \BerkK \smallsetminus \PP^1(K): x \text{ is id-indifferent}\}.
  \]
The proof of the Third Indifference Lemma shows that $I_f$ is open for the
strong topology.

\begin{theorem}
  \label{thm:invertible}
  Let $f \in K(z)$ be an invertible rational function. 
  \begin{enumerate}
  \item\label{item:1} If $f(z) = z$, then $\Fix(f) = \BerkK$.
  \item If $f(z) = z + 1$, then the id-indifference locus for $f$ is 
    \[
      I_f = \{\zeta_{a,r} \in \BerkK : |a| > 1, 1 < r \le |a|\},
    \]
    and $\Fix(f)$ is the closure of $I_f$. (See Figure \ref{fig:1_a}.) In
    particular, $\Fix(f)$ is a single indifferent component containing the
    doubled classical fixed point $\infty$.
  \item If $f(z) = \lambda z$ with $|\lambda| \ne 1$, then $\Fix(f) = \{0, \infty\}$.
  \item If $f(z) = \lambda z$ with $|\lambda| = 1$ and $|\lambda - 1| = 1$, then
    $\Fix(f) = [0,\infty]$ is a single indifferent component.
  \item\label{item:5} If $f(z) = \lambda z$ with $0 < |\lambda - 1| < 1$, then the
    id-indifferent locus for $f$ is
    \[
       I_f = \{\zeta_{a,r} \in \BerkK: |a(\lambda - 1)| < r\},
    \]
    and $\Fix(f)$ is the closure of $I_f$. (See Figure \ref{fig:1_b}.) In
    particular, $\Fix(f)$ is a single indifferent component containing the
    classical fixed points $\{0,\infty\}$.
  \end{enumerate}
  Any other $f$ is conjugate to a function in precisely one of the cases
  \ref{item:1}-\ref{item:5}.
\end{theorem}


\begin{figure}[ht]
    \centering
    \begin{subfigure}[b]{0.45\textwidth}
      \centering

\begin{tikzpicture}[
            scale=1.4583,
            blackdot/.style={
                circle,
                fill=black,
                inner sep=0.8pt,
                minimum size=0pt
            },
            reddot/.style={
                circle,
                fill=red,
                inner sep=0.8pt,
                minimum size=0pt
            },
            bluedot/.style={
                circle,
                fill=blue,
                inner sep=0.8pt,
                minimum size=0pt
            },
            node distance = 1cm
        ]

            \node[bluedot] (0) [label={[font=\footnotesize, xshift=-2pt, yshift=-5pt]above right:{$\zeta_{0,1}$}}] at (0, 0.225) {};


            \node[reddot] (-1) at (0, 0.975) {};  
                \node[bluedot] (-1_1) at (0.211, 0.896) {};
                \node[bluedot] (-1_2) at (0.144, 0.802) {};
                \node[bluedot] (-1_3) at (-0.211, 0.896) {};
                \node[bluedot] (-1_4) at (-0.144, 0.802) {};

            \node[reddot] (-2) at (0, 1.300) {};  
                \node[bluedot] (-2_1) at (0.276, 1.134) {};
                \node[bluedot] (-2_2) at (0.228, 1.072) {};
                \node[bluedot] (-2_3) at (-0.276, 1.134) {};
                \node[bluedot] (-2_4) at (-0.228, 1.072) {};

            \node[reddot] (-3) at (0, 1.625) {};  
                \node[bluedot] (-3_1) at (0.360, 1.409) {};
                \node[bluedot] (-3_2) at (0.297, 1.328) {};
                \node[bluedot] (-3_3) at (-0.360, 1.409) {};
                \node[bluedot] (-3_4) at (-0.297, 1.328) {};

            \node[reddot] (-4) at (0, 1.950) {};  
                \node[bluedot] (-4_1) at (0.444, 1.684) {};
                \node[bluedot] (-4_2) at (0.366, 1.584) {};
                \node[bluedot] (-4_3) at (-0.444, 1.684) {};
                \node[bluedot] (-4_4) at (-0.366, 1.584) {};

            \node[reddot] (-5) at (0, 2.275) {};  
                \node[bluedot] (-5_1) at (0.527, 1.958) {};
                \node[bluedot] (-5_2) at (0.435, 1.840) {};
                \node[bluedot] (-5_3) at (-0.527, 1.958) {};
                \node[bluedot] (-5_4) at (-0.435, 1.840) {};

            \node[reddot] (-6) at (0, 2.600) {};  
                \node[bluedot] (-6_1) at (0.611, 2.233) {};
                \node[bluedot] (-6_2) at (0.504, 2.096) {};
                \node[bluedot] (-6_3) at (-0.611, 2.233) {};
                \node[bluedot] (-6_4) at (-0.504, 2.096) {};

            \node[reddot] (-7) at (0, 2.925) {};  
                \node[bluedot] (-7_1) [label={[font=\footnotesize]right:{$\zeta_{a,1}$}}] at (0.694, 2.508) {};
                \node[bluedot] (-7_2) at (0.572, 2.353) {};
                \node[bluedot] (-7_3) at (-0.694, 2.508) {};
                \node[bluedot] (-7_4) at (-0.572, 2.353) {};

            \node[blackdot] (C_inf) [label={[font=\footnotesize, xshift=-3pt, yshift=-4pt]above right:{$\infty$}}] at (0, 3.825) {};

            \draw[red, dashed, thick] (0) -- (-1);
            \draw[red, thick] (-1) -- (-2);
            \draw[red, thick] (-2) -- (-3);
            \draw[red, thick] (-3) -- (-4);
            \draw[red, thick] (-4) -- (-5);
            \draw[red, thick] (-5) -- (-6);
            \draw[red, thick] (-6) -- (-7);

            \draw[red, dashed, thick] (-7) -- (C_inf);

            \draw[red, thick] (-1) -- (-1_1);
            \draw[red, thick] (-1) -- (-1_2);
            \draw[red, thick] (-1) -- (-1_3);
            \draw[red, thick] (-1) -- (-1_4);

            \draw[red, thick] (-2) -- (-2_1);
            \draw[red, thick] (-2) -- (-2_2);
            \draw[red, thick] (-2) -- (-2_3);
            \draw[red, thick] (-2) -- (-2_4);

            \draw[red, thick] (-3) -- (-3_1);
            \draw[red, thick] (-3) -- (-3_2);
            \draw[red, thick] (-3) -- (-3_3);
            \draw[red, thick] (-3) -- (-3_4);

            \draw[red, thick] (-4) -- (-4_1);
            \draw[red, thick] (-4) -- (-4_2);
            \draw[red, thick] (-4) -- (-4_3);
            \draw[red, thick] (-4) -- (-4_4);

            \draw[red, thick] (-5) -- (-5_1);
            \draw[red, thick] (-5) -- (-5_2);
            \draw[red, thick] (-5) -- (-5_3);
            \draw[red, thick] (-5) -- (-5_4);

            \draw[red, thick] (-6) -- (-6_1);
            \draw[red, thick] (-6) -- (-6_2);
            \draw[red, thick] (-6) -- (-6_3);
            \draw[red, thick] (-6) -- (-6_4);

            \draw[red, thick] (-7) -- (-7_1);
            \draw[red, thick] (-7) -- (-7_2);
            \draw[red, thick] (-7) -- (-7_3);
            \draw[red, thick] (-7) -- (-7_4);

\end{tikzpicture}
      \caption{$f(z)=z+1$}
      \label{fig:1_a}
    \end{subfigure}
    \hfill 
    \begin{subfigure}[b]{0.45\textwidth}
      \centering

        \begin{tikzpicture}[
            blackdot/.style={
                circle,
                fill=black,
                inner sep=0.8pt,
                minimum size=0pt
            },
            reddot/.style={
                circle,
                fill=red,
                inner sep=0.8pt,
                minimum size=0pt
            },
            bluedot/.style={
                circle,
                fill=blue,
                inner sep=0.8pt,
                minimum size=0pt
            },
            node distance = 1cm
        ]


            \node[blackdot] (C_0) [label={[font=\footnotesize, xshift=-2pt, yshift=-4pt]above right:{$0$}}] at (0, -2.250) {};

            \node[reddot] (2) at (0, -1.200) {};        
                \node[bluedot] (2_1) at (0.314, -1.574) {};   
                \node[bluedot] (2_2) at (0.423, -1.444) {};   
                \node[bluedot] (2_3) at (-0.423, -1.444) {};  
                \node[bluedot] (2_4) at (-0.314, -1.574) {};  

            \node[reddot] (1) at (0, -0.750) {};
                \node[bluedot] (1_1) at (0.314, -1.124) {};   
                \node[bluedot] (1_2) at (0.423, -0.994) {};   
                \node[bluedot] (1_3) at (-0.423, -0.994) {};  
                \node[bluedot] (1_4) at (-0.314, -1.124) {};  

            \node[reddot] (0) at (0, -0.300) {};        
                \node[bluedot] (0_1) at (0.314, -0.674) {};   
                \node[bluedot] (0_2) at (0.423, -0.544) {};   
                \node[bluedot] (0_3) at (-0.423, -0.544) {};  
                \node[bluedot] (0_4) at (-0.314, -0.674) {};  


            \node[reddot] (-1) at (0, 0.150) {};
                \node[bluedot] (-1_1) at (0.423, -0.094) {};   
                \node[bluedot] (-1_2) at (0.314, -0.224) {};   
                \node[bluedot] (-1_3) at (-0.423, -0.094) {};  
                \node[bluedot] (-1_4) at (-0.314, -0.224) {};  

            \node[reddot] (-2) at (0, 0.600) {};
                \node[bluedot] (-2_1) [label={[font=\footnotesize]right:{$\zeta_{a,|a(\lambda-1)|}$}}] at (0.423, 0.356) {};
                \node[bluedot] (-2_2) at (0.314, 0.226) {};
                \node[bluedot] (-2_3) at (-0.423, 0.356) {};
                \node[bluedot] (-2_4) at (-0.314, 0.226) {};

            \node[reddot] (-3) at (0, 1.050) {};
                \node[bluedot] (-3_1) at (0.423, 0.806) {};
                \node[bluedot] (-3_2) at (0.314, 0.676) {};
                \node[bluedot] (-3_3) at (-0.423, 0.806) {};
                \node[bluedot] (-3_4) at (-0.314, 0.676) {};

            \node[reddot] (-4) at (0, 1.500) {};
                \node[bluedot] (-4_1) at (0.423, 1.256) {};
                \node[bluedot] (-4_2) at (0.314, 1.126) {};
                \node[bluedot] (-4_3) at (-0.423, 1.256) {};
                \node[bluedot] (-4_4) at (-0.314, 1.126) {};

            \node[reddot] (-5) at (0, 1.950) {};
                \node[bluedot] (-5_1) at (0.423, 1.706) {};   
                \node[bluedot] (-5_2) at (0.314, 1.576) {};   
                \node[bluedot] (-5_3) at (-0.423, 1.706) {};  
                \node[bluedot] (-5_4) at (-0.314, 1.576) {};  

            \node[blackdot] (C_inf) [label={[font=\footnotesize, xshift=-3pt, yshift=-4pt]above right:{$\infty$}}] at (0, 3.000) {};

            \draw[red, dashed, thick] (2) -- (C_0);
            \draw[red, dashed, thick] (-5) -- (C_inf);

            \draw[red, thick] (2) -- (1);
            \draw[red, thick] (0) -- (1);
            \draw[red, thick] (0) -- (-1);
            \draw[red, thick] (-1) -- (-2);
            \draw[red, thick] (-2) -- (-3);
            \draw[red, thick] (-3) -- (-4);
            \draw[red, thick] (-4) -- (-5);

            \draw[red, thick] (2) -- (2_1);
            \draw[red, thick] (2) -- (2_2);
            \draw[red, thick] (2) -- (2_3);
            \draw[red, thick] (2) -- (2_4);

            \draw[red, thick] (1) -- (1_1);
            \draw[red, thick] (1) -- (1_2);
            \draw[red, thick] (1) -- (1_3);
            \draw[red, thick] (1) -- (1_4);

            \draw[red, thick] (0) -- (0_1);
            \draw[red, thick] (0) -- (0_2);
            \draw[red, thick] (0) -- (0_3);
            \draw[red, thick] (0) -- (0_4);

            \draw[red, thick] (-1) -- (-1_1);
            \draw[red, thick] (-1) -- (-1_2);
            \draw[red, thick] (-1) -- (-1_3);
            \draw[red, thick] (-1) -- (-1_4);

            \draw[red, thick] (-2) -- (-2_1);
            \draw[red, thick] (-2) -- (-2_2);
            \draw[red, thick] (-2) -- (-2_3);
            \draw[red, thick] (-2) -- (-2_4);

            \draw[red, thick] (-3) -- (-3_1);
            \draw[red, thick] (-3) -- (-3_2);
            \draw[red, thick] (-3) -- (-3_3);
            \draw[red, thick] (-3) -- (-3_4);

            \draw[red, thick] (-4) -- (-4_1);
            \draw[red, thick] (-4) -- (-4_2);
            \draw[red, thick] (-4) -- (-4_3);
            \draw[red, thick] (-4) -- (-4_4);

            \draw[red, thick] (-5) -- (-5_1);
            \draw[red, thick] (-5) -- (-5_2);
            \draw[red, thick] (-5) -- (-5_3);
            \draw[red, thick] (-5) -- (-5_4);

        \end{tikzpicture}
      \caption{ $f(z) = \lambda z, \; 0 < |\lambda - 1| < 1.$ }
      \label{fig:1_b}
    \end{subfigure}
    \caption{Schematic of the id-indifference locus (red) and additively
      indifferent points (blue) of $\Fix(f)$. A general additively indifferent
      point has been labeled in each figure. The hyperbolic lengths of the
      maximal fixed arcs leading off of $[0,\infty]$ are discussed at the end
      of this section.}
\end{figure}
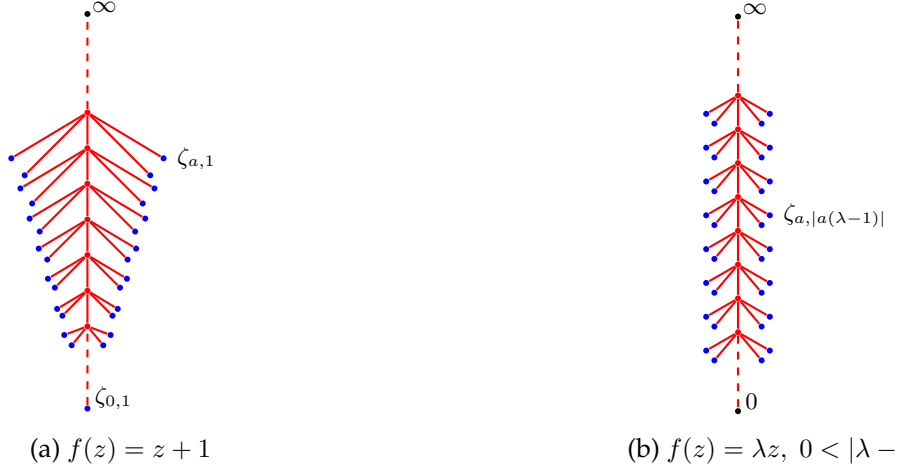


\begin{proof}
If $f$ is the identity map, then evidently $\Fix(f) = \BerkK$. 

If $f$ is not the identity, then $f$ has one or two distinct classical fixed
points. Conjugate one such fixed point to $\infty$ in order to assume that
\[
  f(z) = \lambda z  + \mu, 
  \]
where $\lambda, \mu \in K$ and $\lambda \ne 0$. If $\lambda = 1$, then we may
replace $f$ with
\[
 \mu^{-1}f(\mu z) = z + 1.
\]
If instead $\lambda \ne 1$, then replace $f$ with
\[
   f\big(z + \mu (1-\lambda)^{-1}\big) - \mu (1-\lambda)^{-1} = \lambda z.
\]
Thus, we find that our original $f$ was conjugate to one of the cases listed in
the theorem.

Suppose first that $f(z) = z+1$. The Gauss point is fixed. If $|a| \le 1$ and
$U_a$ denotes the open disk of radius~1 containing $a$, then $f(U_a) =
U_{a+1}$. As $U_{a} \cap U_{a+1} = \varnothing$, the function $f$ does not fix 
any point of the form $\zeta_{a,r}$ with $|a| \leq 1, r < 1$. If $|a| > 1$ and
$|b| \le |a|$, then
\[
g_{a,b}(z) := b^{-1}f(bz + a) - a/b  = z + 1/b.
\]
If $|b| > 1$, then $g_{a,b}$ reduces to the identity. If $|b| = 1$, then
$\tilde g_{a,b}(z) = z + c$ for some $c \in \tilde K^\times$. If $|b| < 1$,
then $g_{a,b}$ has constant reduction. This means $\zeta_{a,r}$ is an
id-indifferent fixed point when $1 < r \le |a|$; it is additively indifferent
if $1 = r < |a|$; and it is not fixed if $r <1$. We conclude that $\Fix(f)$
consists of a single indifferent component that is the closure of $I_f$.

In what remains, we assume that $f(z) = \lambda z$ for some $\lambda \ne 0$.

Suppose that $|\lambda| < 1$. Then $0$ and $\infty$ are attracting and
repelling fixed points, respectively. Since $f(\zeta_{a,r}) = \zeta_{\lambda
  a,|\lambda|r}$, there is no other fixed point. That is, $\Fix(f) =
\{0,\infty\}$.

Next suppose that $|\lambda| > 1$. Replacing $f$ with $1 / f(1/z) =
\lambda^{-1} z$, we find ourselves in the setting of the previous paragraph.

Next suppose that $|\lambda| = 1$ and $|\lambda - 1| = 1$. That is, $\tilde
\lambda \ne 1$. For $b \ne 0$, we have
\[
  b^{-1}f(bz) = \lambda z,
  \]
which shows $f(\zeta_{0,r}) = \zeta_{0,r}$ for all $0 < r < \infty$, and the
tangent map satisfies $T_{\zeta_{0,r}} f(z) = \tilde \lambda z \ne z$. If $|a|
> |b|$, we define
\[
   g_{a,b} := b^{-1}f(bz + a) - a/b  = \lambda z + \frac{a}{b}(\lambda - 1).
\]
Then $g_{a,b}$ has constant reduction, which means $\zeta_{a,|b|}$ is not
fixed.  Type~III and~IV points cannot be isolated fixed points
(Remark~\ref{rem:higher_type_fixed}), so we conclude that $\Fix(f) =
[0,\infty]$.

Finally, suppose that $|\lambda| = 1$ and $|\lambda - 1| < 1$, so that $\tilde
\lambda = 1$. Take $a,b \in K$ with $|a| \ge |b| > 0$. Define
\[
   g_{a,b} := b^{-1}f(bz + a) - a/b  = \lambda z + \frac{a}{b}(\lambda - 1).
\]
If $|a(\lambda-1)| < |b|$, then the reduction of $g_{a,b}$ is the identity. If
instead $|a(\lambda-1)| = |b|$, then $\tilde g_{a,b}(z) = z + c$ for some $c
\in \tilde K \smallsetminus \{0\}$. This means $\zeta_{a,r}$ is an
id-indifferent fixed point if $|a(\lambda - 1)| < r$; it is an additively
indifferent fixed point if $r = |a(\lambda - 1)| \ne 0$; and it is not fixed
otherwise.
\end{proof}

\begin{corollary}
  If $f \in K(z)$ is an invertible rational function, then $\Fix(f)$ has at
  most~2 connected components.
\end{corollary}

In the normalized cases of the theorem, we can say something about how far the
fixed locus extends from the arc $[0,\infty]$ in the strong metric. The
only interesting cases are $f(z) = z+1$ and $f(z) = \lambda z$ with $0 <
|\lambda - 1| < 1$.  For $s > r$, write $\rho(\zeta_{a,r},\zeta_{a,s}) =
\log(s/r)$ for the strong metric on $\BerkK \smallsetminus \PP^1(K)$.

If $f(z) = z+1$ and $\zeta_{a,r}$ is additively indifferent, we see that $r =
1$, and
  \[
    \rho(\zeta_{a,r}, \zeta_{a,|a|}) = \log |a| \to \infty
    \]
as $a \to \infty$. That is, the fixed locus is not contained inside any strong
tube around $[\zeta_{0,1},\infty]$. (Compare Figure~\ref{fig:1_a}.)

If $f(z) = \lambda z$ with $0 < |\lambda - 1| < 1$ and $\zeta_{a,r}$ is
additively indifferent, then $r = |a(\lambda - 1)|$, and 
\[
  \rho(\zeta_{a,r}, \zeta_{a,|a|}) = - \log |\lambda -1|.
\] 
In this case, the fixed locus is contained inside the strong closed tube around
$[0,\infty]$ of radius $-\log |\lambda -1|$. (Compare Figure~\ref{fig:1_b}.)


\section{Indifferent components}
\label{sec:indifferent_components}

Indifferent components have a surprisingly rigid global structure, as
illustrated most strongly by the next result.

\begin{proposition}
  \label{prop:two_fixed_pts}
  Let $f \in K(z)$ be a nonconstant rational function of degree $d$. Each
  indifferent component of the fixed locus of $f$ contains exactly two
  classical indifferent fixed points, counted according to fixed-point
  multiplicity. In particular, the number of indifferent components is at most
  $ \frac{d+1}{2}$.
\end{proposition}

\begin{proof}
  Let $X$ be an indifferent component of the fixed locus. Select a
  distinguished point $x_0 \in X$ for the duration of this argument. All
  classical fixed points are counted according to fixed-point multiplicity and
  pre-images are counted according to their local degree.

  The complement of $X$ is a union of open disks. Let $U$ be such a disk, let
  $x \in X$ be its unique boundary point, and let $\vv \in T_x$ be the
  direction at $x$ pointing toward $U$. Note that $\vv$ cannot be a fixed
  direction, as otherwise there would be an open arc $(x,y) \subset B_x(\vv)^-$
  consisting of fixed points, contradicting the fact that $U$ contains no point
  of $X$. Since there is a non-fixed direction, the tangent map $T_x f$ is not
  the identity. Rumely's First Identification Lemma tells us that
  \[
     F_f(\vv) = s_f(\vv). 
  \]
  As $x$ is indifferent and every direction $\ww$ at $x$ pointing along $X$ is
  fixed, $T_{x}f$ maps $\vv$ to a direction outside $X$. From
  \cite[Prop.~3.10]{Faber_Berk_RamI_2013}, we conclude that $U$ contains $r$
  classical fixed points if and only if it contains $r$ pre-images of the
  distinguished point $x_0$.

  Let $U_1, \ldots, U_n$ be the open disks in the complement of $X$ that
  contain a pre-image of $x_0$. Since $x_0$ is an indifferent fixed point, and
  since every other point of $X$ is fixed, there are exactly $d-1$ pre-images
  of $x_0$ in $\cup_{i} U_i$. By the preceding paragraph, this union contains
  exactly $d-1$ classical fixed points. The remaining two classical fixed
  points of $f$ must lie inside $X$.
\end{proof}

\begin{corollary}
  \label{cor:at_least_3_peaked}
  Let $f \in K(z)$ be a nonconstant rational function. If $x$ is a classical
  fixed point for $f$ with fixed-point multiplicity at least~3, then it lies in
  a peaked component of $\Fix(f)$.
\end{corollary}


The Indifference Lemmas imply that there is a reciprocity between the
fixed-point multipliers of the two fixed points in an indifferent component.

\begin{proposition}
    \label{prop:indifferent_component_types}
  Suppose that $X$ is an indifferent component of $\Fix(f)$ for a rational
  function $f \in K(z)$.
  \begin{enumerate}
    \item If $X$ contains a unique classical fixed point, then every other
      point of $X$ is id-indifferent or additively indifferent.
    \item If $X$ contains two distinct classical fixed points $x \ne y$, then
      the multipliers at $x$ and $y$ satisfy
      \[
      \lambda_x \lambda_y \equiv 1 \pmod{\mm}.
      \]
      If $\lambda_x \equiv 1 \pmod{\mm}$, every other point of $[x,y]$ is
      id-indifferent. If $\lambda_x \not\equiv 1 \pmod{\mm}$, then every other
      point of $X$ is multiplicatively indifferent, and $X = [x,y]$.
  \end{enumerate}
\end{proposition}

\begin{proof}
  Suppose first that $X$ contains a unique classical fixed point, say
  $x$. Since an indifferent component must contain exactly two classical fixed
  points (Proposition~\ref{prop:two_fixed_pts}), we see that $x$
  is a double fixed point and $\lambda_x = 1$. By the First Indifference Lemma, there exists an
  id-indifferent fixed point $y$ in $X$. Let $y' \in X$ be a non-leaf type~II
  point. By the Third Indifference Lemma, the indifference type is locally
  constant along $[y,y']$, and hence constant. Since $y$ is id-indifferent, so
  is every other point in the arc. It follows that every non-leaf point of $X$
  is id-indifferent. Evidently, the leaf points of $X$ are additively
  indifferent points and $x$.

  Now suppose that $X$ contains two distinct classical fixed points, say $x$
  and $y$. The argument in the previous paragraph can be used again to show
  that the indifference type along $(x,y)$ is locally constant, and hence
  constant. By the First Indifference Lemma, we see that every type~II point of
  $(x,y)$ has tangent map $z \mapsto \tilde \lambda_x z$ in some coordinate. If
  $\lambda_x \equiv 1 \pmod{\mm}$, then every point of $(x,y)$ is
  id-indifferent. Otherwise, every point is multiplicatively indifferent and $X
  = [x,y]$.

  It remains to address the claim about the multipliers. Without loss of
  generality, we may choose a coordinate so that $x = 0$ and $y = \infty$. The
  argument in the proof of the First Indifference Lemma shows that we can
  consistently choose coordinates so that $T_{\zeta_{0,r}} f(z) = \tilde
  \lambda_0 z$ for all sufficiently small $r > 0$. For any $r > 0$, the
  direction $\vec{0}$ is fixed at $\zeta_{0,r}$, and the argument in the Second
  Indifference Lemma shows that the fixed-point multiplier for $T_{\zeta_{0,r}}
  f$ at $\vec{0}$ is locally constant. Since we know the value is $\tilde
  \lambda_0$ for $r$ sufficiently small, we see that it is $\tilde \lambda_0$
  for all $r > 0$.

  To address the fixed point at $\infty$, we set $g(z) = 1/f(1/z)$ and use the
  argument in the preceding paragraph again. It follows that the fixed-point
  multiplier for $T_{\zeta_{0,r}} g$ at $\vec{0}$ is $\tilde \lambda_\infty$
  for all $r > 0$. In particular, the reduction at $\zeta_{0,1}$ is $\tilde
  g(z) = \tilde \lambda_\infty z$. By the previous paragraph, we see that
  \[
  \tilde \lambda_\infty z = \tilde g(z)
  = 1/\tilde f(1/z) = \tilde \lambda_0^{-1} z.
  \]
  That is, $\lambda_0 \lambda_\infty \equiv 1 \pmod \mm$, as desired. 
\end{proof}

Now we argue that indifferent components do not exist for polynomial maps or
maps with potential good reduction. This is accomplished by exhibiting certain
special subtrees of $\BerkK$ that an indifferent component must meet.

Fix a rational function $f \in K(z)$. For $S \subset \BerkK$, the
\textbf{connected hull} of $S$ is defined to be the smallest connected subset
of $\BerkK$ containing $S$. For any $y \in \BerkK$, we define a finitely
branching subtree $\Gamma_y$ to be the connected hull of $y$ and all of its
pre-images inside $\BerkK$:
\[
  \Gamma_y = \Hull(\{y\} \cup f^{-1}\{y\}).
\]

Following Rumely \cite[\S5]{Rumely_new_equivariant}, if $x$ is a type~II fixed
point, a direction $\vv \in T_x$ is called a \textbf{shearing direction} for
$f$ if $T_x f(\vv) \ne \vv$ and there exists a classical fixed point for $f$ in
the direction $\vv$. Remarkably, if $x$ has a shearing direction, then $x \in
\Gamma_y$ for every $y \in \BerkK$. In fact, more is true.

\begin{lemma}
  \label{lem:big_intersection}
  Let $f \in K(z)$ be a nonconstant rational function, and let $S$ be the set
  of all fixed points of $f$ with a shearing direction. Then
  \[
    \Hull(S) \subset \bigcap_{y \in \BerkK} \Gamma_y.
  \]
\end{lemma}

\begin{proof}
  To begin, we claim that $Y := \bigcap_{y \in \BerkK} \Gamma_y$ is
  connected. Indeed, if $x,x' \in Y$, let $[x,x']$ be the arc connecting
  them. Since $\Gamma_y$ is connected for each $y$, we find that $[x,x'] \in
  \Gamma_y$. Thus, $[x,x'] \subset Y$, so that $Y$ is connected.
  
  It therefore suffices to show that $S \subset Y$, or equivalently, that $S
  \subset \Gamma_y$ for each $y \in \BerkK$. Let $x \in S$, and let $\vv \in
  T_x$ be a direction such that $\ww = T_x f(\vv)$ is a shearing direction. As
  $\ww$ is a shearing direction, $\vv \neq \ww$ and $x$ is not
  id-indifferent. Writing $D = B_x(\ww)^-$, we find that $f(D) = \BerkK$ by
  Rumely's First Identification Lemma. If $y \in D$, then there is a pre-image
  of $y$ in $B_x(\vv)^-$, implying $x \in \Gamma_y$.  If instead $y \not\in D$,
  then there is a pre-image of $y$ inside $D$ since $f(D) = \BerkK$. Again, we
  find that $x \in \Gamma_y$. We conclude that $x \in \Gamma_y$ for all $y$,
  and the proof is complete.
\end{proof}

\begin{lemma}
  \label{lem:two_bad}
  Let $x$ be a fixed point of type~II or~III for a rational function $f \in
  K(z)$. If there are two bad directions for $f$ at $x$, then $x \in \bigcap_{y
    \in \BerkK} \Gamma_y$.
\end{lemma}

\begin{proof}
  Let $\uu, \vv \in T_x$ be bad directions for $f$. If $y \in B_x(\uu)^-$, then
  $y$ has a pre-image in $B_x(\vv)^-$ since $\vv$ is bad. If $y \not\in
  B_x(\uu)^-$, then $y$ has a pre-image in $B_x(\uu)^-$ since $\uu$ is bad. In
  either case, $x \in \Gamma_y$.
\end{proof}

\begin{proposition}
  \label{prop:indiff_inter}
  Let $f \in K(z)$ be a nonconstant rational function of degree $d \ge 2$. Let
  $X$ be an indifferent component of $\Fix(f)$. Let $X_{\sh}$ be the set of all
  points in $X$ with a shearing direction. Then
  \[
    X \cap \bigcap_{y \in \BerkK} \Gamma_y = \Hull (X_{\sh}) \ne \varnothing.
  \]
\end{proposition}

\begin{proof}
  For brevity, set $Y := \bigcap_{y \in \BerkK} \Gamma_y$.
  
  Since $f$ has $d+1 \ge 3$ classical fixed points,
  Proposition~\ref{prop:two_fixed_pts} shows that there is a classical fixed
  point $y$ lying outside $X$. Let $x$ be the point of $X$ that is closest to
  $y$ (Proposition~\ref{prop:closest_point}). Let $\vv \in T_x$ be the
  direction containing $y$. Then $\vv$ cannot be a fixed direction, else there
  would be points of $X \cap B_x(\vv)^-$ closer to $y$ than $x$ is. It follows
  that $\vv$ is a shearing direction, implying $\Hull(X_{\sh}) \ne
  \varnothing$. Lemma~\ref{lem:big_intersection} implies that $\Hull (X_{\sh})
  \subseteq X \cap Y$. To complete the proof, we need to show $\Hull (X_{\sh})
  \supseteq X \cap Y$, for which it is enough to show that every point of the
  connected closed set $X \cap Y$ with valence at most~1 has a shearing
  direction. Let $x$ be such a point.

  The point $x$ must have at least one bad direction since $d \ge 2$ and
  $\deg_f(x) = 1$. If a bad direction at $x$ is not fixed, then it is a
  shearing direction by Rumely's First Identification Lemma, and we are
  done. So we assume that every bad direction at $x$ is fixed. If $x$ has two
  distinct bad directions $\uu, \vv$, then any point $x' \in X$ sufficiently
  close to $x$ along an arc in $\uu$ or $\vv$ also has two bad directions. Such
  a point $x'$ lies in $Y$ by Lemma~\ref{lem:two_bad}, which means $x$ has
  valence at least~2 in $X \cap Y$, a contradiction. Thus, there is a unique
  bad direction $\vv$ at $x$, and it is fixed. Let $y \in B_x(\vv)^-$. Since
  $\deg_f(x) = 1$, no other direction maps to $\vv$. It follows that every
  pre-image of $y$ lies in $B_x(\vv)^-$, so that $x \not\in \Gamma_y$. This
  contradiction establishes the proposition.
\end{proof}

\begin{corollary}
  \label{cor:no_indifferent}
  Let $f \in K(z)$ have degree at least~2 and a totally ramified fixed
  point. Then $\Fix(f)$ has no indifferent component. In particular, this
  applies when $f$ is a polynomial or has potential good reduction.
\end{corollary}

\begin{proof}
  Let $x$ be a totally ramified fixed point for $f$. Then $\Gamma_x =
  \{x\}$. If $y_0,y_1$ are two $K$-rational points with $y_1 \not\in \{y_0,
  f(y_0)\} \cup f^{-1}(y_0)$, we find $\Gamma_{y_0} \cap \Gamma_{y_1}$ contains
  no $K$-rational point. Thus, $\{x\} = \bigcap_{y \in \BerkK} \Gamma_y$ is not
  a classical fixed point. By Proposition~\ref{prop:indiff_inter}, the point
  $x$ must lie in any indifferent component. Being totally ramified, the local
  degree of $x$ is at least~2, which is absurd.
\end{proof}

We close with a quadratic family of examples that exhibits different types of
components in its fixed locus.  It also serves as an application of the results
in this section.

\begin{example}
  Suppose the field $K$ has characteristic different from~$2$. Recall that we
  write $\cO, \mm$ for the ring of integers and the maximal ideal of $K$,
  respectively. Let $a,b, T \in \cO$ satisfy the following constraints:
  \[
    |a| = |b| = 1, \quad |T| < |2|, \quad \text{and} \quad |2T| < |b + b^{-1} -2| \le 1.
  \]
  Consider the family of quadratic functions
  \[
    f(z) = \frac{z^{2} + a(b+T)z}{(b^{-1}+T)z + a}.
  \]
  The classical fixed points of $f$ are $0, \infty$, and $\alpha
  :=\frac{a(b-1+T)}{b^{-1}-1+T}$. One verifies that $f_u(z) := u^{-1}f(uz)$ has
  reduction $\tilde b z$ for all $u \in K \smallsetminus \{0\}$. It follows
  that the entire arc $[0, \infty]$ consists of indifferent fixed points.
 
  Suppose first that $b = 1 - T$ or $b^{-1} = 1-T$. Then $0$ or $\infty$ is a
  double fixed point, which implies that the component of the fixed locus
  containing the arc $[0,\infty]$ is peaked
  (Corollary~\ref{cor:at_least_3_peaked}). Since $f$ is quadratic, it must have
  potential good reduction, and $\Fix(f)$ consists of a single component by
  Corollary~\ref{cor:no_indifferent}.

  Now suppose that $b \ne 1-T \ne b^{-1}$. The multiplier of the fixed
  point $\alpha$ is 
  \[
     \lambda_{\alpha} = \frac{2T + (b+b^{-1}-2)}{2T + T(b+b^{-1}-2) + T^{2}}.
  \]
  One checks that
  \begin{align*}
    \lambda_\alpha \equiv 1 \pmod \mm & \text{ if } |b+b^{-1} - 2| < |2T| \\    
    |\lambda_\alpha| > 1 & \text{ if } |2T| < |b+b^{-1} - 2| \le 1.
  \end{align*}
  Thus, the fixed point $\alpha$ is indifferent in the first case and repelling
  in the secon.

  Suppose that $|b + b^{-1} - 2| < |2T|$. All three of the classical fixed
  points are indifferent. By Proposition~\ref{prop:two_fixed_pts}, $f$ must
  have a peaked component and hence must have potential good reduction. In
  Corollary~\ref{cor:total_no_of_nonclassical_comp} below, we will find that
  $\Fix(f)$ is connected.

  Finally, suppose that $|b + b^{-1} - 2| > |2T|$. The existence of a repelling
  classical fixed point implies that $f$ does not have potential good
  reduction, and the component of the fixed locus containing $[0,\infty]$ is
  indifferent. If $b \equiv 1 \pmod \mm$, then the open arc $(0,\infty)$
  consists of id-indifferent fixed points
  (Proposition~\ref{prop:indifferent_component_types}). If $b \not \equiv 1
  \pmod \mm$, then the arc $(0,\infty)$ consist of multiplicatively indifferent
  fixed points (Proposition~\ref{prop:indifferent_component_types}). Note that
  $f$ cannot have a peaked component since it does not have potential good
  reduction. It follows from Corollary~\ref{cor:no_indifferent} that $\Fix(f)$
  consists of the isolated component $\{\alpha\}$ and the indifferent component
  containing $[0,\infty]$.
\end{example}



\section{Classical fixed points in a non-classical component}
\label{sec:counting_classical}

Let $f \in K(z)$ be a nonconstant rational function, let $X$ be a non-classical
component of $\Fix(f)$, and let $x \in X$ be a type~II point. For $\vv \in
T_x$, write
\begin{itemize}
  \item $F_f(\vv,X)$ for the number of classical fixed points in $X \cap
    B_x(\vv)^-$, counted according to fixed-point multiplicity, and
  \item $\Ncf(f,x)$ for the number of critically fixed directions of $T_x f$.
\end{itemize}

\begin{lemma}
  \label{lem:classical_fixed}
  Let $f$ be a nonconstant rational function. Let $X$ be a non-classical
  component of $\Fix(f)$, and let $x \in X$ be a type~II point that is not
  id-indifferent. For a direction $\vv \in T_x$ that is not critically fixed,
  the set $X \cap B_x(\vv)^-$ contains finitely many type~II repelling fixed
  points, and
  \[
  F_f(\vv,X) = \tilde F_f(\vv) +
  \sum_{\substack{y \in X \cap B_x(\vv)^- \\ \text{type~II repelling}}}
  \left( \deg_f(y) - 1 - \Ncf(f,y)\right),
  \]
  where $\tilde F_f(\vv)$ was defined in \S\ref{sec:rumely}. 
\end{lemma}

\begin{proof}
  If the direction $\vv$ is not fixed under the map $T_x f$, then $X \cap
  B_x(\vv)^- = \varnothing$, and the lemma follows trivially. So we assume that
  $\vv$ is fixed and non-critical under $T_{x}f$.  The set $B_x(\vv)^-
  \smallsetminus X$ is a disjoint union of open disks $\{U_\alpha\}$. Let
  $\uu_\alpha$ be the tangent vector at the unique boundary point of $U_\alpha$
  that points into $U_\alpha$. Then $\{\uu_\alpha\}$ is the set of all tangent
  vectors at boundary points of $X \cap B_x(\vv)^-$ that point out of
  $X$. Using the Jump Number Formula (\S\ref{sec:rumely}) and Rumely's First
  Identification Lemma, we see that
  \begin{align*}
    F_f(\vv) &= \tilde F_f(\vv) + s_f(\vv) \\
    &= \tilde F_f(\vv) + \sum_{y \in B_x(\vv)^-} \Big(\deg_f(y) -
    \deg_f(y,\ww_x)\Big) \\
    &= \tilde F_f(\vv) + \sum_{y \in X \cap B_x(\vv)^-}  \Big(\deg_f(y) -
    \deg_f(y,\ww_x)\Big) \\
    &\phantom{= \tilde F_f(\vv)}+ \sum_\alpha \sum_{y \in U_\alpha} \Big(\deg_f(y) -
    \deg_f(y,\ww_x)\Big)\\
    &= \tilde F_f(\vv) + \sum_{y \in X \cap B_x(\vv)^-}  \Big(\deg_f(y) -
    \deg_f(y,\ww_x)\Big)
    + \sum_\alpha s_f(\uu_\alpha).
  \end{align*}

  As $x,y \in X$ and $X$ is connected, the arc $[y,x]$ lies entirely in
  $X$. This arc is a representative of the direction $\ww_x$; since it consists
  entirely of fixed points, we have $\deg_f(y, \ww_x) = 1$. Applying this
  observation to the above calculation, we obtain
  \begin{equation}
    \label{eq:ball_surplus}
    F_f(\vv) = \tilde F_f(\vv) +
    \sum_{\substack{y \in X \cap B_x(\vv)^- \\ \text{type~II repelling}}} \left(\deg_f(y) - 1\right)
    + \sum_\alpha s_f(\uu_\alpha).
  \end{equation}
  Since $s_f(\uu_\alpha) \ge 0$, we find
  \begin{align*}
  \#\{ y \in X \cap B_x(\vv)^- : y \text{ type~II repelling}\} &\le
  \sum_{\substack{y \in X \cap B_x(\vv)^- \\ \text{type~II repelling}}} \left(\deg_f(y) - 1\right) \\
  &\le F_f(\vv) - \tilde F_f(\vv) < \infty. 
  \end{align*}

  Fix an index $\alpha$, and let $y$ be the unique boundary point of
  $U_\alpha$. If $\uu_\alpha$ is a non-fixed direction for $T_y f$, then
  Rumely's First Identification Lemma shows that $F_f(\uu_\alpha) =
  s_f(\uu_\alpha)$.  If instead $\uu_\alpha$ is a fixed direction, then it must
  be true that $\deg_f(y,\uu_\alpha) > 1$. For otherwise, there would be an arc
  of fixed points $(y, y')$ inside $U_\alpha$, contradicting the fact that
  $\uu_\alpha$ points out of $X$. Thus, $\uu_\alpha$ is a critically fixed
  direction for $T_y f$. This can only occur if $y$ is a type~II repelling
  fixed point. It follows that $F_f(\uu_\alpha) = 1 + s_f(\uu_\alpha)$. 
  Applying these observations to \eqref{eq:ball_surplus}, we see that
  \begin{align*}
    F_f(\vv) &= \tilde F_f(\vv) + \sum_{\substack{y \in X \cap B_x(\vv)^-
        \\ \text{type~II repelling}}} \left(\deg_f(y) - 1 - \Ncf(f,y)\right) +
    \sum_\alpha F_f(\uu_\alpha).
  \end{align*}
  The proof concludes upon observing that
  \[
     F_f(\vv,X) = F_f(\vv) - \sum_\alpha F_f(\uu_\alpha). \qedhere
  \]
\end{proof}

Finally, we recall and prove Theorem~A.

\begin{theoremA}
  Let $f \in K(z)$ be a nonconstant rational function, and let $X$ be a
  non-classical component of $\Fix(f)$. Then $X$ contains finitely many type II
  repelling fixed points, and the number of classical fixed points inside $X$,
  counted according to fixed-point multiplicity, is
  \[
     2 + \sum_{\substack{x \in X \\ \text{type~II repelling}}} \left(\deg_f(x) - 1 - \Ncf(f,x)\right).
  \]
\end{theoremA}

\begin{proof}
  Finiteness of the set of type~II repelling fixed points in $X$ follows directly
  from Lemma~\ref{lem:classical_fixed}.
  
  Suppose first that $X$ is an indifferent component. Then $X$ has precisely~2
  classical fixed points by Proposition~\ref{prop:two_fixed_pts}, and the
  formula in the theorem holds.

  Next suppose that $X = \{x\}$ is an isolated repelling type~II fixed
  point. Then every fixed direction in $T_x$ is critical. This means
  $\Ncf(f,x) = \deg_f(x) + 1$, and the formula in the theorem equals zero, as
  expected of a component with no classical fixed point.

  Finally, suppose that $X$ is a peaked component, but not a singleton. Select
  a type~II repelling fixed point~$x$. For a critically fixed direction $\vv
  \in T_x$ , there is an arc $(x,y) \subset B_x(\vv)^-$ on which $f$ expands
  lengths. In particular, $X \cap B_x(\vv)^- = \varnothing$. For $\vv$ a
  non-fixed direction, it is clear that $X \cap B_x(\vv)^- =
  \varnothing$. Lemma~\ref{lem:classical_fixed} shows that the number of
  classical fixed points in $X$, counted with fixed-point multiplicity, is 
  \begin{align*}
  \sum_{\substack{\vv \in T_x \\ \text{fixed, non-critical}}} &
  \left[\tilde F_v(\vv) +
  \sum_{\substack{y \in X \cap B_x(\vv)^- \\ \text{type~II repelling}}}
  \left( \deg_f(y) - 1 - \Ncf(f,y)\right)\right] \\
  &= \sum_{\substack{\vv \in T_x \\ \text{fixed, non-critical}}}
  \tilde F_v(\vv) +
  \sum_{\substack{y \in X \smallsetminus \{x\} \\ \text{type~II repelling}}}
  \left( \deg_f(y) - 1 - \Ncf(f,y)\right).
  \end{align*}
  Since the rational function $T_x f$ has $\deg_f(x) + 1$ fixed points, the
  first sum is
  \[
  \sum_{\substack{\vv \in T_x \\ \text{fixed, non-critical}}}
  \tilde F_v(\vv)
  = \deg_f(x) + 1 - \Ncf(f,x) = 2 + \big(\deg_f(x) - 1 - \Ncf(f,x)\big). \qedhere
  \]
\end{proof}


\section{Hyperbolic components}
\label{sec:hyperbolic}

For a rational function $f \in K(z)$, a component of $\Fix(f)$ that does not
contain a classical fixed point will be called \textbf{hyperbolic}. This
terminology is justified by the fact that such a component resides entirely
inside $\BerkK \smallsetminus \PP^1(K)$, which is often called the ``Berkovich
hyperbolic space''. We will see below that such components are relatively rare. 

\begin{proposition}
  \label{prop:hyperbolic_hull}
  Let $f \in K(z)$ be a nonconstant rational function, and let $X$ be a
  hyperbolic component of $\Fix(f)$. Then $X$ is the connected hull of its
  repelling type~II fixed points. In particular, $X$ contains no id-indifferent
  or additively indifferent fixed point.
\end{proposition}

\begin{proof}
  Let $I_X$ and $A_X$ be the set of id-indifferent and additively indifferent
  fixed points of $X$, respectively. The Third Indifference Lemma shows that
  $I_X$ is an open subset of $X$, while the Second Indifference Lemma implies
  that $A_X$ lies in the boundary of $I_X$. Assume that $I_X \cup A_X =
  \varnothing$. Then $X$ consists of multiplicatively indifferent fixed points
  and repelling type~II fixed points. Every branch point $x \in X$ is a fixed
  point with at least~3 fixed directions, so $x$ must be repelling. By the
  Third Indifference Lemma, every leaf of $X$ must be repelling too, and the
  proposition follows.  Hence, let us assume that $I_X \cup A_X \ne
  \varnothing$. In particular, we must have $I_X \ne \varnothing$.

  Let $X' = X \smallsetminus (I_X \cup A_X)$. Let $T \subset X'$ be a connected
  component, and write $k$ for the number of repelling fixed points in $T$. By
  definition, $X'$ consists of multiplicatively indifferent fixed points and
  repelling type~II fixed points. The argument in the previous paragraph
  applied to $T$ instead of $X$ shows $T$ is a tree with $k$ repelling fixed
  points as vertices, and we have
  \[
    \sum_{x \in T} \left(\deg_f(x) - 1 - \Ncf(f,x)\right) 
    = -2k + \sum_{\substack{x \in T,\\ x \text{ repelling}}} \left(\deg_f(x) + 1 - \Ncf(f,x)\right).
    \]
  Write $\alpha(T)$ for the quantity on the left.  The quantity in the final
  sum is nonnegative, being the number of non-critical fixed directions at $x$,
  counted according to fixed-point multiplicity. As $T$ is a tree with $k$
  vertices, the sum of the valences at its vertices is $2(k-1)$. By the Second
  Indifference Lemma, the valence at a vertex $x \in T$ is precisely the number
  of non-critical simple fixed directions at $x$. Thus,
  \begin{align*}
    \alpha(T) &= -2k + 2(k-1) + \sum_{x \in T}
    \#\left\{\text{fixed directions at $x$ with multiplicity $> 1$}\right\} \\
    &= -2 + \sum_{x \in T}
    \#\left\{\text{fixed directions at $x$ with multiplicity $> 1$}\right\}.
  \end{align*}
  As $I_X \ne \varnothing$ and $X$ is connected, $T$ must contain a repelling
  type~II fixed point $x$ in the boundary of $I_X$. At $x$, a direction meeting
  $I_X$ is fixed with multiplicity at least~2. That is, $\alpha(T) \ge 0$.

  Summing $\alpha(T)$ over all connected components of $X'$ and applying
  Theorem~A, we find that $X$ contains at least~2 classical fixed points, a
  contradiction. We are forced to conclude that $X$ has no id-indifferent or
  additively indifferent fixed point, as desired. In particular, this means $X
  = T$, and so $X$ is the connected hull of its type~II repelling fixed points.
\end{proof}

\begin{corollary}
  If a component of $\Fix(f)$ contains an id-indifferent fixed point, then it
  contains a classical fixed point.
\end{corollary}

We are now in a position to give a simple proof that there are finitely many
non-classical components of $\Fix(f)$. Write $\Gamma_\Fix$ for the connected
hull of the classical fixed points of a rational function $f \in K(z)$.

\begin{corollary}
  \label{cor:components_meet_GammaFix}
  Let $f \in K(z)$ be a nonconstant rational function. Each component of
  $\Fix(f)$ contains a leaf or a branch point of $\Gamma_\Fix$. In particular,
  $\Fix(f)$ has finitely many components.
\end{corollary}

\begin{proof}
  Let $X$ be a component of $\Fix(f)$. The leaves of $\Gamma_\Fix$ are the
  classical fixed points, so $X$ contains a leaf of $\Gamma_\Fix$ if and only
  if $X$ is not hyperbolic. Suppose that $X$ is hyperbolic, and let $x$ be a
  repelling point of $X$. By Proposition~\ref{prop:hyperbolic_hull}, there is
  no id-indifferent point in $X$. By the Second Indifference Lemma, the fixed
  point multiplicity of any fixed direction at $x$ is one. As $x$ is repelling,
  there are at least three fixed directions at $x$. By Rumely's First
  Identification Lemma, $x$ is a branch point of $\Gamma_{Fix}$. The final
  statement of the lemma follows from the fact that $\Gamma_\Fix$ is a finite
  tree, and hence has finitely many leaves and branch points.
\end{proof}

The local degrees at the type~II repelling fixed points of a hyperbolic
component must satisfy a strong compatibility condition. 

\begin{proposition}
  \label{prop:hyperbolic_formula}
  Let $f \in K(z)$ be a nonconstant rational function, and suppose that $X$ is
  a hyperbolic component of $\Fix(f)$. Let $x_1, \ldots, x_n$ be the type~II
  repelling fixed points in $X$. We have
  \[
  \sum_{i=1}^n \deg_f(x_i) = n - 1 \quad \text{in $\tilde K$}.
  \]
\end{proposition}

\begin{proof}
  For the duration of this proof, we write $\equiv$ for the equality relation
  in $\tilde K$.

  By Proposition~\ref{prop:hyperbolic_hull}, $X$ is the connected hull of $x_1,
  \ldots, x_n$, and $X$ contains no id-indifferent fixed point. Note that if
  $y$ is a branch point of $X$, then it has at least~3 fixed directions. It
  follows that $y$ is type~II and repelling. Thus, we can view $X$ as a graph with
  vertex set $x_1, \ldots, x_n$. Write $v(x_i)$ for the valence of $x_i$ in
  $X$.

  Let $x = x_i$ be a vertex. The map $T_x f$ has degree $\deg_f(x)$, and
  consequently $\deg_f(x) + 1$ fixed points. Each non-critically fixed
  direction at $x$ corresponds to an edge $e$ of $X$; write $\lambda_{x,e} \in
  \tilde K^\times$ for the associated fixed-point multiplier of $T_x f$. As $X$
  contains no id-indifferent fixed point, the Second Indifference Lemma implies
  that $\lambda_{x,e} \ne 1$. The number of critically fixed directions at $x$
  is $\deg_f(x) + 1 - v(x)$. The fixed-point index formula
  \cite[Thm.~1.14]{Silverman_Dynamics_Book_2007} for $T_x f$ yields
  \begin{equation}
    \label{eq:vertex_index}
     \deg_f(x) + 1 - v(x) + \sum_{e \ni x} \frac{1}{1-\lambda_{x,e}} \equiv 1. 
  \end{equation}

  Now let $e$ be an edge of $T$, and let $x,y$ be the two vertices incident on
  $e$. The Fourth Indifference Lemma shows that $\lambda_{x,e} \, \lambda_{y,e}
  \equiv 1$. It follows that
  \begin{equation}
    \label{eq:index_reciprocity}
    \frac{1}{1-\lambda_{x,e}} + \frac{1}{1- \lambda_{y,e}} \equiv 1.
  \end{equation}
  Summing \eqref{eq:vertex_index} over all vertices $x = x_i$ gives
  \[
     \sum_{i=1}^n \deg_f(x_i) - \sum_{i=1}^n v(x_i) + \sum_e \sum_{x_i \in e}
     \frac{1}{1-\lambda_{x_i,e}} \equiv 0.
  \]
  As $T$ is a tree with $n$ vertices, it has $n-1$ edges. Applying
  \eqref{eq:index_reciprocity} and noting that the sum of all valences of all
  vertices is twice the number of edges, we obtain
  \[
  \sum_{i=1}^n \deg_f(x_i) - 2(n-1) + \sum_e 1
  = \sum_{i=1}^n \deg_f(x_i) - (n-1) \equiv 0. \qedhere
   \]
\end{proof}

\begin{proposition}
  \label{prop:connected_fixed_locus}
  The fixed locus of a rational function $f$ of degree $d \geq 2$ is connected if
  and only if
  \[
    \sum_{\substack{x \text{ type~II} \\f(x) = x}} \left(\deg_f(x) - 1 - \Ncf(f,x)\right) = d-1.
  \]
\end{proposition}

\begin{proof}
  Let $X_1, \ldots, X_n$ be all of the non-classical components of $\Fix(f)$
  (Corollary~\ref{cor:components_meet_GammaFix}), and let $y_1, \ldots, y_m$ be
  the isolated classical fixed points for $f$. The number of classical fixed
  points in some $X_i$ is therefore $d + 1 - m$. Summing the formula in
  Theorem~A over all components $X_i$ gives
  \begin{align*}
    d + 1 - m &= \sum_{i=1}^n \Big(2 + \sum_{\substack{x \in X \\ \text{type~II repelling}}} \left(\deg_f(x) - 1 - \Ncf(f,x)\right)\Big) \\
    &= 2n + \sum_{\substack{x \text{ type~II} \\f(x) = x}} \left(\deg_f(x) - 1 - \Ncf(f,x)\right);
  \end{align*}
  the switch to a sum over all type~II fixed points is justified by
  Remark~\ref{rem:sum_over_all_pts}. Thus, we have the formula
  \begin{equation}
    \label{eq:thmA_all_comps}
    \sum_{\substack{x \text{ type~II} \\f(x) = x}} \left(\deg_f(x) - 1 - \Ncf(f,x)\right) = d + 1 - m - 2n. 
  \end{equation}

  Suppose that $\Fix(f)$ is connected. Then there cannot be an isolated fixed
  point, and there must be only one non-classical component. That is, $m = 0$
  and $n = 1$, and the result follows from \eqref{eq:thmA_all_comps}.

  Conversely, suppose that the sum on the left side of
  \eqref{eq:thmA_all_comps} equals $d-1$. Simplifying \eqref{eq:thmA_all_comps}
  shows that $2n - 2 + m = 0$. If $n = 0$, then $f$ has only $m = 2$ classical
  fixed points, a contradiction. If $n > 1$, then $2n - 2 + m > 0$, another
  contradiction. So we must have $n = 1$ and $m = 0$, which is equivalent to
  saying that $\Fix(f) = X_1$ is connected.
\end{proof}



We close this section with some examples of rational functions with hyperbolic
components. 

\begin{example}[Simple hyperbolic component]
  Suppose that $K$ has residue characteristic $p > 0$. Fix a degree $d \ge p$
  and an element $t \in K$ with $|t| < 1$. Set
  \[
  f(z) = \begin{cases}
    z^d & \text{ if $d = p$} \\
    tz^d + z^p & \text{ if $d > p$}.
  \end{cases}
  \]
  Then $x = \zeta_{0,1}$ is fixed, and every fixed direction at $x$ is
  critically fixed. It follows that $\{x\}$ is a hyperbolic component of $\Fix(f)$. 
\end{example}

\begin{example}[Hyperbolic segment]
  Again, suppose that $K$ has residue characteristic $p > 2$. Fix a degree $d
  \ge p+1$, and fix an element $t \in K$ with $|t| < 1$. Define
  \[
    f(z) = \frac{t z^{d-p+2} + 2z^2}{-t^{2p-1} z^d + 2 tz^{d-p+1} + 4z - 2}.
  \]
  One can prove directly that $f$ has degree~$d$ by substituting a nonzero root
  of the numerator into the denominator and verifying that the result is
  nonzero. Now
  \[
     \tilde f(z) = \frac{z^2}{2z-1},
  \]
  which has a unique non-critically fixed point at $\infty$. If we define
  $g(z) = t^2 f(z/t^2)$, then 
  \[
     \tilde g(z) = \frac{z}{2 - z^{p-1}},
  \]
  which has a unique non-critically fixed point at $0$. One verifies that the
  reduction of $f_u(z) = u^{-1}f(uz)$ is $\frac{1}{2} z$ for $1 < |u| <
  |t|^{-2}$. It follows that the entire segment $[\zeta_{0,1},
    \zeta_{0,|t|^{-2}}]$ consist of fixed points. From the calculations above,
  there is a small strong neighborhood of this segment that contains no other
  fixed point. Hence, the arc $[\zeta_{0,1}, \zeta_{0,|t|^{-2}}]$ forms a
  hyperbolic component of $\Fix(f)$.
\end{example}



\section{Crucial points and fixed points}
\label{sec:crucial_weight}

Let $x \in \BerkK$ be a fixed point for a nonconstant rational function $f \in
K(z)$. Recall that a direction $\vv \in T_x$ is a \textbf{shearing direction}
for $f$ if there exists a classical fixed point of $f$ in the direction $\vv$,
but $T_x f (\vv) \ne \vv$. We write $\Nshear(f,x)$ for the number of shearing
directions at $x$. Rumely's \textbf{crucial weight} is the function $\wR \colon
\BerkK \to \NN$ defined as follows:
\begin{enumerate}
\item If $x \in \PP^1(K)$,  then $\wR(x) = 0$.
\item If $x$ is of type~II, III, or IV and is fixed by $f$, then $\wR(x) =
  \deg_f(x) - 1 + \Nshear(f,x)$.
\item If $x$ is of type~II, III, or IV and is not fixed by $f$, write $v(x)$
  for the number of directions in $T_x$ that contain a type~I fixed point. Then
     $\wR(x) = \max\{0, v(x) - 2\}$. 
\end{enumerate}
A point $x \in \BerkK$ with $\wR(x) > 0$ is called a \textbf{crucial
  point}.

For a component $X$ of $\Fix(f)$, we define
\[
  \alpha(X) = \sum_{x \in X} \left(\deg_f(x) - 1 - \Ncf(f,x)\right).
  \]
By Theorem~A, the sum is finitely supported, and a non-classical component $X$
contains $2 + \alpha(X)$ classical fixed points.

\begin{lemma}
  \label{lem:alpha_sum}
  Let $f \in K(z)$ be a rational function of degree $d \ge 1$. Write $c$ for
  the number of attracting or repelling classical fixed points of $f$, and let
  $X_1,\ldots, X_n$ be a sequence of non-classical components of
  $\Fix(f)$. Suppose that every non-hyperbolic component is among the
  $X_i$. Then
  \[
     \sum_{i=1}^n \alpha(X_i) = d+1 - c - 2n.
  \] 
\end{lemma}

\begin{proof}
  Sum the formula in Theorem~A for $X_1, \ldots, X_n$ and add $c$. This gives
  the total number of classical fixed points of $f$, counted according to
  fixed-point multiplicity, which is $d+1$.
\end{proof}

We recall the main result of this section:

\begin{theorem*}[Rumely's Weight Formula {\cite[Thm.~6.1]{Rumely_new_equivariant}}]
  If $f \in K(z)$ is a rational function of degree $d \ge 1$, then
  \[
    \sum_{x \in \BerkK} \wR(x) = d - 1.
  \]
\end{theorem*}

\begin{proof}
  Let $p_1, \ldots, p_c$ be the attracting or repelling classical fixed points
  of $f$. Let $b_1, \ldots, b_r$ be the branch points of the graph
  $\Gamma_\Fix$ that are not fixed by $f$; these are precisely the non-fixed
  points of $\BerkK$ with positive crucial
  weight. Corollary~\ref{cor:components_meet_GammaFix} asserts that there are
  finitely many components of $\Fix(f)$; let $\{X_1, \ldots, X_n\}$ be all of
  the non-classical components.

  Define a real tree $\cT$ as follows. Let $\tilde \cT \subset \BerkK$ be the
  connected hull of
  \[
   \{p_1, \ldots, p_c, b_1, \ldots, b_r\} \cup \bigcup_{i=1}^n X_i.
  \]
  Define $\cT$ to be the quotient of $\tilde \cT$ given by collapsing each
  $X_i$ to a point, which we call $x_i$. Define
  \[
    V(\cT) = \{x_1, \ldots, x_n, p_1, \ldots, p_c, b_1, \ldots, b_r\} \subset \cT.
  \]
  Every classical fixed point is in $\tilde cT$, so $\Gamma_\Fix \subset \tilde
  cT$, and $\cT$ agrees with the corresponding quotient of $\Gamma_\Fix$. By
  Corollary~\ref{cor:components_meet_GammaFix}, every leaf and branch point of
  $\Gamma_\Fix$ descends to an element of $V(\cT)$ under the quotient, so
  $V(\cT)$ is a vertex set for $\cT$.

  We now compute the valence of each vertex.  At a type~II point $x \in X_i$,
  we can partition the directions in $T_x$ into four sets:
  \begin{itemize}
    \item shearing directions: these contain a classical fixed point by
      definition;
    \item critically fixed directions: these contain a classical fixed point by
      Rumely's First Identification Lemma;
    \item non-critically fixed directions: these contain an arc internal to
      $X_i$; and
    \item non-fixed non-shearing directions: these contain no classical fixed
      point and no arc internal to $X_i$.
  \end{itemize}
  Consequently, the valence of $X_i$ as a vertex of $\cT$ is
  \begin{align*}
    v(X_i) &= \sum_{x \in X_i} \left(\Nshear(f,x) + \Ncf(f,x)\right) \\ &=
    \sum_{x \in X_i} \left( \deg_f(x)-1 + \Nshear(f,x)\right) - \left(\deg_f(x)
    - 1 - \Ncf(f,x)\right) \\ &= \sum_{x \in X_i} \wR(x) - \alpha(X_i).
  \end{align*}
  Evidently, the valences of the other vertices in $\cT$ are given by:
  \[
    v(p_i) = 1 \qquad \text{and} \qquad v(b_i) = 2 + \wR(b_i).
  \]

  The number of edges in the tree $\cT$ is $n + c + r - 1$. The sum of the
  valences at all vertices of a graph is twice the number of edges, so we
  obtain
  \begin{align*}
    2(n+c+r-1) &= c + \sum_{i=1}^n \left( \sum_{x \in X_i} \wR(x) -
    \alpha(X_i)\right) + \sum_{i=1}^r \left(2 + \wR(b_i)\right) \\ &= c + 2r +
    \sum_{x \in S} \wR(x) - \sum_{i=1}^n \alpha(X_i),
  \end{align*}
  where $S = \{b_1, \ldots, b_r\} \cup \bigcup_{i=1}^n X_i$. Substituting the
  expression in Lemma~\ref{lem:alpha_sum} and massaging shows that
  \begin{equation*}
    \label{eq:maybe_partial_sum}
    \sum_{x \in S} \wR(x) = d-1.
  \end{equation*}
  By definition of the crucial weight, every crucial point lies in $S$. 
\end{proof}

We close this section with several consequences of Rumely's Weight Formula and
our results.

\begin{corollary}
    \label{cor:total_no_of_nonclassical_comp}
  Let $f \in K(z)$ be a nonconstant rational function of degree $d \ge
  2$. Every non-classical component of $\Fix(f)$ contains a crucial point. In
  particular, the number of non-classical components is bounded by $d-1$.
\end{corollary}

\begin{proof}
  Let $X$ be a non-classical component of $\Fix(f)$. Suppose first that $X$ is
  a peaked component. Then it contains a repelling fixed point $x$, and $\wR(x)
  \ge \deg_f(x) - 1 \ge 1$.  Now suppose that $X$ is an indifferent
  component. From Proposition~\ref{prop:indiff_inter}, there exists a point $x$
  in $X$ with a shearing direction, and
  \[
    \wR(x) = \deg_f(x) - 1 + \Nshear(f,x) = \Nshear(f,x) \ge 1.
  \]
  Rumely's Weight Formula shows there are at most $d-1$ crucial points, and
  hence at most $d-1$ non-classical components of $\Fix(f)$.
\end{proof}

The proof of Theorem~B requires Rumely's Weight Formula, so we have delayed it
until now to avoid the appearance of circular logic.

\begin{proof}[Proof of Theorem~B]
  Write $p \ge 0$ for the residue characteristic of $K$. We must show that
  $\Fix(f)$ has no hyperbolic component if $p = 0$ or $p > d =
  \deg(f)$. Suppose otherwise, and let $X$ be a hyperbolic component with
  repelling type~II fixed points $x_1,\ldots, x_n$.

  Suppose first that $p = 0$. Since $\deg_f(x_i) \ge 2$ for each $i$, we see
  that
  \[
     \sum_{i=1}^n \deg_f(x_i) - (n-1) \ge n + 1.
  \]
  But Proposition~\ref{prop:hyperbolic_formula} asserts that the left side of
  this inequality vanishes, which is a contradiction.

  Now suppose that $p > d$. 
  \begin{equation}
    \label{eq:rumely_largep}
  \sum_{i=1}^n \deg_f(x_i) - (n-1) = 1 + \sum_{i=1}^n \left (\deg_f(x_i) - 1\right)
  \le 1 + \wR(x_i) \le d,
  \end{equation}
  where the final inequality is a consequence of Rumely's Weight Formula. Since
  $p > d$, the residue characteristic cannot divide the left side of
  \eqref{eq:rumely_largep}, which contradicts
  Proposition~\ref{prop:hyperbolic_formula}. This completes the proof. 
\end{proof}


\bibliographystyle{amsalpha}
\bibliography{fixed_locus}

\end{document}